\documentclass[10pt]{article}

\usepackage{amsfonts,amssymb,amsmath}
\usepackage{dsfont}
\input{xy}
\usepackage[all]{xy}
\xyoption{all}
\xyoption{poly}

\newtheorem{thm}{Theorem}[section]

\newtheorem{lem}[thm]{Lemma}
\newtheorem{pro}[thm]{Proposition}
\newtheorem{defi}[thm]{Definition}%[section]
\newtheorem{rem}[thm]{Remark}%[section]
\newtheorem{exa}[thm]{Example}

\newenvironment{proof}{\noindent \textbf{{Proof.}} \sf}

\def\B{{\mathcal B}}

\def\F{{\mathcal F}}

\def\lim{\mathop{\rm lim}\nolimits}

\def\Hom{\mathop{\sf Hom}\nolimits}

\def\Ext{\mathop{\sf Ext}\nolimits}

\begin{document}

\sf

\title{A finite dimensional algebra of the diagram of a knot}
\author{Claude Cibils}
%\thanks{}

\date{}

\maketitle

\begin{abstract}
To a regular projection of a knot we associate a finite dimensional non-commutative associative algebra which is self-injective and special biserial.

\end{abstract}


\section{\sf Introduction}

Let $D$ be an oriented knot diagram, that is a regular projection of a knot to the plane where each crossing is given with the information of which
part of the knot under-crosses the other one; this appears in $D$ by interrupting the part of the diagram which is under-crossing.

In this note we associate to $D$ a finite dimensional algebra over a field $k$,  presented
by a quiver with relations deduced from the knot. In other words we associate to $D$ a $k$-category given by a presentation; it has
a finite number of objects corresponding to the crosses, and  finite dimensional vector spaces of morphisms.
The algebra  of the diagram is the direct sum of all the morphisms of the $k$-category with product induced
by the composition of the category or, equivalently, is the path algebra of the quiver modulo the two-sided ideal generated by the relations.

This algebra of the knot projection is Morita reduced,  special biserial and self-injective.
As such it is not invariant under Reidemeister moves since its dimension changes.

The main purpose of this note is the description of  this family of algebras. They can be of interest in order to test homological conjectures
or to analyze  representation theory aspects. On the other hand one may expect information on the knot via the algebra of a diagram,
 as well as  links between knot theory and the representation theory of finite dimensional algebras.

On some examples this algebra happens to admit a connected grading by the fundamental group of the knot.

Acknowledgements: I thank the organizers and the lecturers of the CIMPA research school "Symplectic Geometry and Geometric Topology"
 Mekn\`{e}s (Morocco) in 2012, which was inspiration to the content of this note.

\section{Quiver and relations of an oriented knot diagram}

A quiver is an oriented graph $Q$ given by two finite sets, $Q_0$ (vertices) and $Q_1$ (arrows) and two maps $s,t : Q_1\to Q_0$
assigning to each arrow $a$ the  source and target vertices $s(a)$ and $t(a)$.
Consider the vector space $kQ$  with basis the set of oriented paths of $Q$ including the trivial ones given by the vertices.
The quiver path algebra is $kQ$  equipped with by the product on paths induced by their concatenation if it can be performed and $0$ otherwise.
This way the vertices provide a complete set of primitive orthogonal idempotents. Note that $Q_0$ needs to be non empty in order to get an algebra.

Consider $F$ the two sided ideal generated by the arrows. Clearly $kQ/F\simeq kQ_0$ where $kQ_0$
is the commutative semi-simple algebra of the set $Q_0$. In case $Q$ has no oriented cycles $F$ is the Jacobson radical of $kQ$.
A well known theorem of P. Gabriel states that for each  finite dimensional Morita reduced algebra
$\Lambda$ over an algebraically closed field there is a unique quiver $Q_\Lambda$ such that $\Lambda$ is isomorphic to a
quotient $kQ_\Lambda/I$ where $I$ is an admissible two-sided ideal, that is  $I$ contains $F^2$ and is
contained in  $F^n$ for some positive integer $n$. Note that $I$ is not unique in general, while $Q_\Lambda$
 is the $\Ext$ quiver with vertices the iso-classes of simple modules and as many arrows between two simples
 than the dimension of $\Ext_\Lambda^1$ between them.

Let $D$ be an oriented knot diagram.
An \textbf{arc} of $D$ is obtained by following the diagram according to its orientation from an under-crossing to the next one, in other words
 an arc is a connected
component of $D$. A \textbf{segment} of the diagram is obtained by following the diagram according to its orientation, from one crossing to a crossing
with no crossings in-between.

\begin{defi}
\textbf{The quiver $\mathbf {Q_D}$} has set of vertices the crossings of $D$. Note that this requires to have at least one crossing in $D$ in order
that  $(Q_D)_0\neq \emptyset$.
Each segment provides an arrow having source and target the corresponding crossings.
We say that the source of a an arrow is negative or positive according if the segment starts respectively
by under-crossing or by over-crossing. Similarly the target vertex of an arrow is also negative or positive.
\end{defi}

\begin{rem}\label{fund}
There are two specific oriented cycles at each vertex $e$ as follows:
\begin{enumerate}
\item
 $\alpha_e$ starts at $e$ by the arrow with
positive source $e$  and successive arrows by browsing the oriented diagram until reaching the arrow having target $e$ with negative sign,
\item Similarly, $\beta_e$  starts at $e$ but by the negative source arrow and ends with the arrow with positive target $e$.
\end{enumerate}
\end{rem}

\begin{defi}
 The  two \textbf{fundamental} \textbf{cycles} at $e$ are as follows: the over-crossing  (or positive) one  $\gamma_e^+=\beta_e \alpha_e$
  and the under-crossing
  (or negative) one  $\gamma_e^-=\alpha_e\beta_e$.
\end{defi}

\begin{lem}
The length of the fundamental cycles do not depend on the vertex and equals the number of arrows $n_D$ of $Q_D$.
\end{lem}

Let $\tau : Q_0\to k^\bullet$ be a map. For instance $\tau(e)=q^{l(e)}$ where $q\in k^\bullet$ and $l:Q_0\to \mathbb{Z}$ is some  map
 given for example by the length of $\alpha_e$.
We consider a  two-sided ideal $I_\tau$ of $kQ_D$  generated by two kind of relations as follows.\begin{itemize}
\item
Type I : Let $ba$ be a path of length $2$ in $Q_D$ and let $e=t(a)=s(b)$ be its middle vertex.
 In case the signs at $e$ of $a$ and $b$ are different (that is if $ba$ is not a follow-up of two segments of the diagram) then $ba$ is a generator.
  This way each vertex of the quiver provides  two generators.
\item
Type II :
The elements $\alpha_e\beta_e - {\tau(e)} \beta_e\alpha_e$ for all the vertices $e$.
\end{itemize}

\begin{defi}
The \textbf{algebra of the diagram} $D$ with respect to $\tau$ is $\Lambda_{D,\tau} = kQ_D/I_\tau$.
\end{defi}

\begin{pro} \label{basis}
The algebra $\Lambda_{D,\tau}$ is finite dimensional.
A basis is given by the positive fundamental cycles and the non-zero paths of length strictly less than $n_D$, that is the paths made by following-up
the segments and having strictly less than $n_D$ segments.

\end{pro}
\begin{proof}
The proof relies on the  observation that a path  $\delta$ of length $n_D+1$ is zero in the algebra.
Indeed  let $\delta = a\gamma$ where $\gamma$ is a fundamental cycle.
If the source sign of $a$ is not the same than the target sign of $\gamma$ then $\delta\in I_\tau$.
Otherwise let $\gamma'$ be the other fundamental cycle, then
$a\gamma$ and $a\gamma'$ are equal up to a non-zero element of $k$ and  the latter is in the ideal.
\end{proof}

\begin{rem}
The number of vertices of $Q$ change through the first and second Reidemeister moves.
\end{rem}

Let $J_D$ be the two-sided ideal generated by the relations of type I and type II', where type II' is the
set of all  paths of length $n_D+1$.
Instead of $\Lambda_{D,\tau}$ a monomial algebra $\Xi_D =kQ_D/J_D$ can be considered which we call the
\textbf{monomial algebra of the diagram}. A larger basis than before is provided by all the non-zero paths of length strictly less than $n_D+1$.

\begin{exa}
Let $D$ be the diagram of the trivial knot with one crossing. Then the algebra $\Lambda_D,\tau$ is $k\{a,b\}/<ab-qba>$ for $q\in k^\bullet$.
This algebra provides the first example for a negative answer to Happel's question,
see \cite{bugrmaso}. More precisely its global dimension is infinite but for $q$ not a root of unity it has zero
Hochschild cohomology in degrees large enough (in fact starting at degree 3). Nevertheless this algebra  verifies Han's conjecture, see
\cite{H}, namely its Hochschild homology is non zero in arbitrarily large degrees (in fact in all degrees).
\end{exa}

\section{Properties}
We recall first the definition of a family of algebras which arose in representation theory of finite dimensional algebras.

\begin{defi}
Let $Q$ be a quiver and $I$ a two sided admissible ideal generated by a set of relation  $\rho$. Then $(Q, \rho)$ is called {\textbf{special biserial}}
(see for instance  \cite{ri})  if it verifies the following conditions:
\begin{enumerate}
 \item Any vertex of $Q$ is the source of at most two arrows and is the target of at most two arrows.
 \item If two different arrows $c$ and $d$ start at the target of an arrow $a$ then at least one of the paths $ca$ or $da$ is in $\rho$.
 \item If two different arrows $a$ and $b$ end at the source of an arrow $c$ then at least one of the two paths $ca$ or $cb$ is in $\rho$
\end{enumerate}

\end{defi}

As mentioned by C.M. Ringel in \cite{ri} special biserial algebras were first considered by I.M. Gelfand and V.A.  Ponomarev in \cite{gepo}.
Blocks of a group algebra with cyclic or dihedral defect group are special biserial.
As a consequence of \cite{gepo} special biserial algebras are of tame representation type, in other words their indecomposable modules can be classified,
 see also \cite{wawa,dosk}. Precise conditions are given in \cite{assembustamalemeur}
 for the vanishing of the first Hochschild cohomology of a special biserial algebras (which in turn implies that the cohomology in degrees larger than $1$
 also vanishes).

 The following result is immediate:

 \begin{pro}
 The algebra (or the monomial algebra) of the diagram of a knot is special biserial.
 \end{pro}

 We recall that an algebra is \textbf{self-injective} if it admits a non degenerate bilinear
 form $\beta : \Lambda\times\Lambda\to k$ which is associative, that is
  $\beta(xy,z)=\beta(x,yz)$ for any triple of elements $(x,y,z)$ in $\Lambda$. Associated to $\beta$ there is a linear map $t:\Lambda\to k$
  given by $t(x)=\beta(x,1)$ which is a
  free generator of the left $\Lambda$-module $\Hom_k(\Lambda, k)$.

  \begin{thm}
  Algebras of diagrams of knots are self-injective.
  \end{thm}

\begin{proof}
Let $\delta$ be a positive length basis path  of $\Lambda_{D,\tau}$  (according to Lemma \ref{basis}) with source vertex $e$. We define  $\delta'$ to be
the path such that
$\delta'\delta =\gamma_e^\epsilon$, where $\gamma_e^\epsilon$ is the fundamental cycle and  $\epsilon$ is the sign of $e$ at the first arrow of $\delta$.
Note that $\left(\gamma^+_e\right)'=e$. Moreover for each vertex $e$ we put  $e'=\gamma_e^+$.

In case $\delta_1$ and $\delta_2$ are basis paths such that  $\delta_2\neq\delta_1'$ we put  $\beta(\delta_2,\delta_1)=0$.

If  $\delta_2 = \delta_1'$ we consider two cases :
\begin{enumerate}
\item In case $\delta_1$ is a vertex or if the source of the first arrow of $\delta_1$ is positive, then $\beta(\delta_1', \delta_1)=1$.
\item If  the sign of the source $e$ of the first arrow of $\delta_1$ is negative, then $\beta(\delta_1', \delta_1)=\tau(e)$.
\end{enumerate}

The only difficulty for verifying  that $\beta$ is associative arises when $\delta_1$ has a first arrow with negative source $e$. We need to prove that
$$\beta(\delta_1', \delta_1)=\beta(\delta_1'\delta_1, e).$$
We have defined $\beta(\delta_1', \delta_1)=\tau(e)$ while
$$\beta(\delta_1'\delta_1, e)=\beta(\gamma_e^-,e)=\beta (\tau(e)\gamma_e^+,e)=\tau(e)\beta(\gamma_e^+,e)=\tau(e)$$
There is no difficulty for showing that $\beta$ is non-degenerated.
\end{proof}

\begin{rem}
The class of special biserial self-injective algebras has been studied by K. Erdmann and A. Skowro\'nski with respect to Euclidian components
of the stable Auslander-Reiten quiver, see \cite{ersk}. See also the work by  Z. Pogorzaly \cite{po} concerning stable equivalence of this class of algebras.
Precise computations of the Hochschild cohomology of certain self-injective special biserial algebras are performed in \cite{snta}.
\end{rem}

\section{Gradings}

The fundamental group \emph{\`{a} la Grothendieck} of a $k$-category has been considered in
\cite{CRS ANT 10, CRS DOC 11, CRS ART 12, CRS PAMS 12, CRS 12}.
Previously a fundamental group depending on a presentation by a quiver with relations has been studied in relation with
representation theory, see for instance \cite{boga,MP,bu1,buca,le,le1,le2}. The main tool for the theory of the intrinsic fundamental
group are the connected gradings as follows.

\begin{defi}
Let $\B$ be a small $k$-category. A \textbf{grading} $X$ of $\B$ with structural group $\Gamma(X)$ is firstly a direct sum decomposition of each vector space
 of morphisms indexed by elements of $\Gamma(X)$ -- a direct summand with index $g$ of this decomposition is
 called a homogeneous component of degree $g$,
 and a non zero morphism in this component is said to be homogeneous of degree $g$.  Secondly the composition
 of two homogeneous morphisms is homogeneous with degree the product of the degrees.
 \end{defi}

The precise definition of homogeneous walks is given for instance in \cite{CRS 12}. Roughly each homogeneous
 morphism $\varphi$ of degree $g$ provides a virtual one $(\varphi,-1)$ with reversed source and target vertices and of settled degree $g^{-1}$.
 A \textbf{homogeneous walk} is a sequence
of concatenated (virtual or not) homogeneous morphisms; its degree is the product of the degrees.

The grading $X$ is \textbf{connected} if between two fixed objects and for any element of $g\in\Gamma(X)$ there exist a homogenous walk
relying the objects having degree $g$. In this case the smash product provides a connected category which is a Galois covering of $\B$, see \cite{CM}.

The fundamental group is obtained by considering all the connected gradings of $\B$ and coherent families of elements of the structure groups with
respect to morphisms of gradings, see \cite{CRS 12}.

We recall the following convention: a crossing of $D$ is positive if following the diagram according to the orientation  the under line
of the diagram goes from right to left. It is negative otherwise.

 Let $D$ be the diagram of an oriented knot. Recall that the fundamental group of the complement of the knot
  with base point at the infinity is generated by
 the loops which passes just under each portion of the knot corresponding to an arc of the projection.

 \begin{defi}
 Let $Q_D$ be the diagram of a knot. The grading of $Q_D$ by the fundamental group of the knot
 is as follows: given an arrow, consider the corresponding segment
 and the arc to which it belongs.
 In case the crossings at the vertices of the arrow have the same sign, the degree of the arrow is the generator of the fundamental group
 corresponding to the arc. Otherwise the degree is trivial.
 \end{defi}

The verification of the following result is not difficult using the Wirtinger presentation of the fundamental group of the knot.

 \begin{exa}
 For the usual diagrams of the trefoils knots or of the figure-height knot the two-sided ideal $I_\tau$ is homogeneous
 and the resulting grading for the algebra of the diagram  is connected.
\end{exa}

Nevertheless it seems that for the diagram of the  knot  6\_\ 3 (see the for instance the Knot Atlas) the ideal is not homogeneous.

%%%%%%%%%%%%%%%%%%%%%%%%%%%%%%%%%%%%%%%%%%%%%%%%%%%%%%%%%%%%%%%%%%%%%%%%%%%%%%%%%%%%%%%%%%%

\footnotesize
\noindent
\\Institut de math\'{e}matiques et de mod\'{e}lisation de Montpellier I3M,\\
UMR 5149\\
Universit\'{e}  Montpellier 2,
\\F-34095 Montpellier cedex 5,
France.\\
{\tt Claude.Cibils@math.univ-montp2.fr}

\end{document}